\documentclass[a4paper,12pt,headsepline]{article}
\usepackage[english]{babel}
\usepackage{amsmath}
\usepackage{amsfonts, amsthm}
\usepackage{mathtools}
\usepackage[utf8]{inputenc}
\usepackage{amssymb}
\usepackage{ushort}
\usepackage{tikz}
\usepackage{enumerate}
\usepackage{multicol}
\usepackage[hyperfootnotes=false]{hyperref}
\usepackage{setspace}
\usepackage{MnSymbol}
\usepackage{array}
\usepackage{caption}
\usepackage{makecell}
\usepackage{url}

\usepackage{geometry}
\allowdisplaybreaks

\definecolor{airforceblue}{rgb}{0.36, 0.54, 0.66}
\definecolor{amethyst}{rgb}{0.6, 0.4, 0.8}
\definecolor{aqua}{rgb}{0.0, 1.0, 1.0}
\definecolor{atomictangerine}{rgb}{1.0, 0.6, 0.4}
\definecolor{bananayellow}{rgb}{1.0, 0.88, 0.21}
\definecolor{bittersweet}{rgb}{1.0, 0.44, 0.37}
\definecolor{blue-violet}{rgb}{0.54, 0.17, 0.89}
\definecolor{bostonuniversityred}{rgb}{0.8, 0.0, 0.0}
\definecolor{brilliantlavender}{rgb}{0.96, 0.73, 1.0}
\definecolor{carminered}{rgb}{1.0, 0.0, 0.22}
\definecolor{ceruleanblue}{rgb}{0.16, 0.32, 0.75}
\definecolor{coquelicot}{rgb}{1.0, 0.22, 0.0}
\definecolor{deepcerise}{rgb}{0.85, 0.2, 0.53}
\definecolor{deepmagenta}{rgb}{0.8, 0.0, 0.8}
\definecolor{deepskyblue}{rgb}{0.0, 0.75, 1.0}
\definecolor{egyptianblue}{rgb}{0.06, 0.2, 0.65}
\definecolor{electricviolet}{rgb}{0.56, 0.0, 1.0}

\usepackage{tikz}
\usetikzlibrary{fit,positioning,shapes,calc}

\tikzstyle{rootN}  =   [circle,scale=0.5]
\tikzstyle{rootA1}  =   [circle,scale=0.5,ball color=deepcerise]
\tikzstyle{rootB1} =   [circle,scale=0.5,ball color=deepcerise]
\tikzstyle{rootA2}  =   [diamond,scale=0.5,ball color=aqua]
\tikzstyle{rootB2}  =   [diamond,scale=0.5,ball color=aqua]
\tikzstyle{rootA3}  =   [star,scale=0.5,ball color=electricviolet]
\tikzstyle{rootB3}  =   [star,scale=0.5,ball color=electricviolet]
\tikzstyle{cluster} = [draw=black,thin,rounded corners,inner sep=\clustersep]

\def\clustersep{3pt}

\def\clusterpicture{\begin{tikzpicture}[node distance=0pt]}
\def\endclusterpicture{\end{tikzpicture}}

\pgfdeclareradialshading{sphere1}{\pgfpoint{0cm}{-0.2cm}}%
{rgb(0cm)=(1,1,1);rgb(0.7cm)=(0.7,0.1,0); rgb(1cm)=(0.5,0.05,0); rgb(1.05cm)=(1,1,1)}
\pgfdeclareradialshading{sphere2}{\pgfpoint{-0.3cm}{0.3cm}}%
{rgb(0cm)=(1,1,1);rgb(0.7cm)=(0.7,0.1,0); rgb(1cm)=(0.5,0.05,0); rgb(1.05cm)=(1,1,1)}
\pgfdeclareradialshading{sphere3}{\pgfpoint{0.3cm}{0.3cm}}%
{rgb(0cm)=(1,1,1);rgb(0.7cm)=(0.7,0.1,0); rgb(1cm)=(0.5,0.05,0); rgb(1.05cm)=(1,1,1)}

\def\Root(#1)#2(#3){
	\path(#1) node[root#2](#3){};
}
\def\RootL(#1,#2)(#3)#4{
	\Root(#1,#2)(#3)
	\node (#3n) [color=purple!50!black,below=of #3.south,scale=0.5,xshift=0pt,yshift=-6pt,font=\ttfamily, anchor=north] {\smash{\raise-4pt\hbox{#4}}};  
}
\def\Cluster(#1)#2{
	\node[cluster,fit=#2] (#1) {};
}
\def\ClusterL(#1)#2#3{
	\Cluster(#1){#2}
	\node (#1n) [color=black,below=of #1.south east, font=\ttfamily, scale=0.5, anchor=north west, xshift=-5pt]{\smash{#3}};
}

\def\frob(#1)(#2){\path[draw,-](#1)--(#2){}}

\def\ecluster(#1)(#2)#3{
	\path(#1) coordinate (#2a){}; 
	\path($(#1)+(0.3,0)$) coordinate (#2c){}; 
	\path($(#1)+(#3/2,0)-(0,#3)$) coordinate (#2t); 
	\path[draw,-,thick] (#2a)--($(#2a)-(0,0.15)$);
	\path[draw,-,thick] (#2a) edge[out=90,in=90] (#2c);
	\path[draw,-,thick] (#2c) edge[out=270,in=180] (#2t);
	\expandafter\edef\csname lastnode#2\endcsname{#2t}
	\expandafter\let\csname cltype#2\endcsname\epsilon
}

\def\ocluster(#1)(#2)#3{
	\path(#1) coordinate (#2a); 
	\path($(#1)-(0,0.1)$) coordinate (#2b); 
	\path($(#1)+(#3/2,0)-(0,#3)$) coordinate (#2t); 
	\path[draw,-,thick] ($(#2a)+(0,0.15)$)--(#2b);
	\path[draw,-,thick] (#2b) edge[out=270,in=180] (#2t);
	\expandafter\edef\csname lastnode#2\endcsname{#2t}
	\expandafter\let\csname cltype#2\endcsname\omega
}

\def\uclustergen(#1)(#2)#3#4{
	\path($(#1)+(0.1,0)$) coordinate (#2c){}; 
	\path($(#2c)-(0,0.1)$) coordinate (#2d){}; 
	\path($(#2c)+(#3/2,0)-(0,#3)$) coordinate (#2t); 
	\path[draw,-,thick] ($(#2c)+(0,0.15)$)--(#2d);
	\path[draw,-,thick] (#2d) edge[out=270,in=180] (#2t);
	\path($(#1)$) coordinate (#2a){}; 
	\path($(#2a)-(0,0.1)$) coordinate (#2b){}; 
	\path($(#2d)+(#3/2,0)-(0,#3)$) coordinate (#2u); 
	\path[draw,-,thick#4] ($(#2a)+(0,0.15)$)--(#2b);          
	\path[draw,-,thick#4] (#2u) edge[out=180,in=270] (#2b);   
	\expandafter\edef\csname lastnode#2\endcsname{#2t}
	\expandafter\let\csname cltype#2\endcsname\upsilon
}

\def\ucluster(#1)(#2)#3{\uclustergen(#1)(#2){#3}{}}
\def\uclusterbr(#1)(#2)#3{\uclustergen(#1)(#2){#3}{,shorten >=1.1pt}}

\def\clusterNWlabel(#1)#2{
	\path(#1a) node[color=blue,anchor=east] {$#2$}; 
}
\def\clusterSWlabel(#1)#2{
	\path(#1t) node[color=blue,anchor=east] {$#2$}; 
}

\def\subroot #1-#2 #3{
	\edef\lastnode{\csname lastnode#1\endcsname}
	\path($(\lastnode)+(#3,0)$) node[root](#2){};
	\path[draw,-,thick](\lastnode)--(#2){};
	\expandafter\def\csname lastnode#1\endcsname{#2}
}

\def\subrootname #1-#2 #3#4{
	\subroot #1-#2 {#3}
	\node (#2n) [color=purple!50!black,below=of #2.south,scale=0.7,yshift=0pt,font=\ttfamily, anchor=north] {$#4$};  
}

\def\endcluster #1 #2{
	\edef\lastnode{\csname lastnode#1\endcsname}
	\path($(\lastnode)+(#2,0)$) node[coordinate](#1end){};
	\path[draw,-,thick](\lastnode)--(#1end){};
	\expandafter\ifx\csname cltype#1\endcsname\upsilon
	\path[draw,-,thick]($(\lastnode)-(0,0.1)$)--($(#1end)-(0,0.1)$){};  
	\else\fi
	\expandafter\def\csname lastnode#1\endcsname{#1end}
}

\def\subcluster #1 #2-#3 (#4,#5){
	\edef\lastnode{\csname lastnode#2\endcsname}
	\expandafter\ifx\csname cltype#2\endcsname\upsilon
	\expandafter\ifx #1u
	\uclusterbr($(\lastnode)+(#4,0)$)(#3){#5}               
	\else
	\csname#1cluster\endcsname($(\lastnode)+(#4,0)$)(#3){#5}
	\fi
	\else
	\csname#1cluster\endcsname($(\lastnode)+(#4,0)$)(#3){#5}
	\fi
	\path[draw,-,thick](\lastnode)--(#3a){};
	\expandafter\ifx\csname cltype#2\endcsname\upsilon
	\path[draw,-,thick]($(\lastnode)-(0,0.1)$)--($(#3a)-(0,0.1)$){};  
	\expandafter\ifx\csname cltype#3\endcsname\upsilon      
	\path($(#3c)+(1.1pt,0)$) node[coordinate] (#3e) {};
	\path($(#3d)+(1.1pt,0)$) node[coordinate] (#3f) {};
	\path[draw,-,thick] (#3a)--($(#3e)-(2.2pt,0)$) {};  
	\path[draw,-,thick] (#3b)--(#3f) {};  
	\expandafter\def\csname lastnode#2\endcsname{#3e}
	\else
	\expandafter\def\csname lastnode#2\endcsname{#3a}
	\fi 
	\else
	\expandafter\def\csname lastnode#2\endcsname{#3a}
	\fi
}

\theoremstyle{definition}
\newtheorem{defn}{Definition}[section]
\newtheorem{defs}[defn]{Definitions}

\theoremstyle{plain}

\newtheorem{thm}[defn]{Theorem}

\newtheorem{prop}[defn]{Proposition}
\newtheorem{lem}[defn]{Lemma}

\newtheorem{rem}[defn]{Remark}
\newtheorem{ex}[defn]{Example}

\theoremstyle{remark}
\newtheoremstyle{case}{}{}{}{}{}{:}{ }{}
\theoremstyle{case}

\DeclareFontFamily{U}{wncy}{}
\DeclareFontShape{U}{wncy}{m}{n}{<->wncyr10}{}
\DeclareSymbolFont{mcy}{U}{wncy}{m}{n}
\DeclareMathSymbol{\Sh}{\mathord}{mcy}{"58} 
\newcommand\blfootnote[1]{%
	\begingroup
	\renewcommand\thefootnote{}\footnote{#1}%
	\addtocounter{footnote}{-1}%
	\endgroup
}

\newcommand{\Spec}{\textrm{Spec}}

\newcommand{\QQ}{\mathbb Q}
\newcommand{\NN}{\mathbb N}
\newcommand{\ZZ}{\mathbb Z}

\newcommand{\RR}{\mathbb R}

\newcommand{\cs}{\mathfrak s}
\newcommand{\ct}{\mathfrak t}
\newcommand{\cp}{\mathfrak p}
\newcommand{\cR}{\mathfrak R}
\newcommand{\regX}{\mathcal X}
\newcommand{\stY}{\mathcal Y}
\newcommand{\nerJ}{\mathcal J}

\newcommand{\disc}{\textrm{disc}}
\newcommand{\ord}{\textrm{ord}}

\usepackage{tcolorbox} \newcommand{\numcirclemodi}[1]{{\setlength\fboxrule{0pt}\fbox{\tcbox[colframe=blue,colback=white,shrink tight,boxrule=0.5pt,extrude by=.7mm]{\small #1}}}}
\usetikzlibrary{matrix,fit}

\title{Differential Forms on Hyperelliptic Curves with Semistable Reduction}
\author{Sabrina Kunzweiler}
\date{}
\begin{document}
\maketitle
\begin{abstract}
	ABSTRACT. Let $C$ be a hyperelliptic curve over a local field $K$ with odd residue characteristic, defined by some affine Weierstraß equation $y^2=f(x)$. We assume that $C$ has semistable reduction and denote by $\regX \rightarrow \Spec\, \mathcal{O}_K$ its minimal regular model with relative dualising sheaf $\omega_{\regX/ \mathcal{O}_K}$. We show how to directly read off a basis for $H^0(\regX,\omega_{\regX/\mathcal{O}_K})$ from the cluster picture of the roots of $f$. Furthermore we give a formula for the valuation of $\lambda$ such that $\lambda \cdot \frac{dx}{2y} \land \dots \land x^{g-1}\frac{dx}{2y}$ is a generator for $\det H^0(\regX,\omega_{\regX/\mathcal{O}_K})$. 
	\blfootnote{\textit{Date:} May 2, 2020.\\2010 \textit{Mathematics Subject Classification.} Primary: 11G20, 14H25; Secondary: 11G40.\\\textit{Key words and phrases.} Hyperelliptic curves, semistable reduction.}
\end{abstract}


\section{Introduction}
Let $K$ be a local field with ring of integers $R$,  normalised valuation $v$ and uniformiser $\pi$. For an element $r \in R$, we write $\bar{r}$ for the reduction modulo $\mathfrak{p}$. The separable closure of $K$ is denoted by $K^{sep}$ and the algebraic closure by $\bar{K}$. We assume that the residue field $k = R/(\pi)$ is perfect and has characteristic $p \neq 2$.

Let $C/K$ be a hyperelliptic curve given by a Weierstraß equation 
\[C : y^2 =f(x).\] 
Throughout this paper, we will always assume that the curve $C$ has semistable reduction and genus $g > 1$. Since $p \neq 2$, it is easy to read off from the polynomial $f$ whether $C$ is semistable and determine the semistable reduction. This is described in \cite{bouw2017computing} and \cite{dokchitser2018arithmetic}.

To a Weierstraß equation, we associate the differential forms
\begin{align*}
\omega_0 = \frac{dx}{2y},\; \omega_1 = \frac{xdx}{2y},\; \dots, \; \omega_{g-1} = \frac{x^{g-1}dx}{2y}.
\end{align*}
These form a basis of $H^0(C, \Omega_{C/K})$. Let $\regX \rightarrow \Spec R$ be the minimal regular model of $C$ and $\omega_{\regX/R}$ the relative dualising sheaf \cite[Definition 6.4.18.]{liu2002alggeo}. In our situation, the relative dualising sheaf is isomorphic to the canonical sheaf \cite[Theorem 6.4.32]{liu2002alggeo}. We have that $H^0(C, \Omega_{C/K}) = H^0(\regX,\omega_{\regX/R}) \otimes_R K$ and $H^0(\regX,\omega_{\regX/R})$ is a free $R$-module of rank $g$. So $\det H^0(\regX,\omega_{\regX/R}) := \bigwedge^g H^0(\regX,\omega_{\regX/R})$ is free of rank one over $R$. 

In this paper, we are going to study the following two problems:
\begin{enumerate}
	\item Let $ \omega := \omega_0 \land \dots \land \omega_{g-1} \in  \det H^0(C,\Omega^1_{C/K})$. Determine $\lambda_C \in K$, such that $\lambda_C \cdot \omega$ generates  $\det H^0(\regX,\omega_{\regX/R})$ as an $R$-module.
	\item Explicitly determine a basis for the global sections of $\omega_{\regX/R}$.
\end{enumerate}

Our approach is based on results of \cite{kausz1999discriminant}. Under simplified hypotheses  a formula for $\lambda_C$ and a description of a basis for the global sections of $\omega_{\regX/R}$ is given in Proposition 5.5. of that paper. 

\subsection{Motivation} 
Our motivation for studying the differential forms on the minimal regular model comes from the Birch and Swinnerton-Dyer conjectures. Originally formulated for elliptic curves, the conjectures were later generalised to abelian varieties over number fields by Tate \cite{tate1966bsd}.

 Let $C$ be a hyperelliptic curve defined over the rational numbers $\QQ$ and let $J$ denote its Jacobian. The second Birch and Swinnerton-Dyer conjecture in this situation is

Here $L(J, s)$ is the $L$-series of $J$ and $r$ its analytic rank. Reg denotes the regulator of $J(\QQ)$. For a prime $p$, the Tamagawa number is denoted by $c_{p}$. The Shafarevich-Tate group is represented by $\Sh(J,\QQ)$ and $J(\QQ)_{\mathrm{tors}}$ is the torsion subgroup of $J(\QQ)$. The results of this paper can be applied to calculate the sixth quantity, that is the period $\Omega$. For the description of this quantity we follow the outline in \cite[Section 3]{van2017numerical} and \cite[Section 3.5]{flynn2001empirical}. Both papers provide numerical evidence for the Birch and Swinnerton-Dyer conjectures for hyperelliptic curves.

Let $\nerJ \rightarrow \Spec \ZZ$ denote the N\'eron model of $J$. We have that $H^0(\nerJ,\Omega_{\nerJ/\ZZ})$ is a free $\ZZ$-module of rank $g$. Let $(\mu_0, \dots, \mu_{g-1})$ be a basis for this module. Then $\mu:= \mu_0 \land \dots \land \mu_{g-1}$ is a generator for $\bigwedge^g H^0(\nerJ,\Omega_{\nerJ/\ZZ})$ and $\Omega$ is defined as 
\[\Omega:= \int_{J(\RR)} |\mu|. \]
Finding a basis for $H^0(\nerJ,\Omega_{\nerJ/\ZZ})$ can be done locally. Let $p \in \ZZ$ be a prime. We write $R=\ZZ_p$ and use the notation introduced in the beginning. Since $H^0(\nerJ,\Omega_{\nerJ/R})$ and $H^0(\regX, \omega_{\regX/R})$ are isomorphic as $\ZZ_p$-modules, see \cite[Lemma 9]{van2017numerical}, it is enough to find a basis for the global sections of $\omega_{\regX/R}$.

For computational purposes it is not always necessary to know the basis of regular differentials. In \cite{van2017numerical} and \cite{flynn2001empirical} the authors first evaluate $\int_{J(\RR)} |\omega|$, where $\omega$ is the exterior product of the differentials $\omega_0, \dots, \omega_{g-1}$ associated to the Weierstraß equation for $C$ and then compute a correction term in order to get the value  $\Omega:= \int_{J(\RR)} |\mu|$. 

\subsection{Results}

We give a brief overview of our main results and illustrate them by means of an example. The results are stated in terms of cluster pictures. A cluster picture is a combinatorial object associated to the Weierstraß equation of a curve. It encodes different invariants of the curve. The concept has been studied in \cite{dokchitser2018arithmetic}. For definitions we refer to \cite[Section 1.5.]{dokchitser2018arithmetic} or Section \ref{sec:cluster} of this paper. Here we only recall the necessary notation.
\begin{multicols}{2}
	\begin{itemize}
		\item[$c_f$] leading coefficient of $f$
		\item[$\cR$]  set of roots of $f$ in $K^{sep}$
		\item[$\cs$] a cluster 
		\item[$z_{\cs}$] a centre of $\cs$
		\item[$d_{\cs}$] depth of $\cs$
		\item[$\delta_{\cs}$] relative depth of $\cs$
		\item[$\cs \land \cs'$] smallest cluster containing $\cs$ and $\cs'$
		\item[$\nu_{\cs}$]  $ = v(c_f)+\sum_{r \in \cR} d_{r \land \cs}$
	\end{itemize}
\end{multicols}

First, we show how to read off a basis for $H^0(\regX,\omega_{\regX/R})$ from the cluster picture of a curve.

\begin{thm}[Theorem \ref{thm:mainbasis}]
	\label{thm:mainbasis_int}
	Let $C/K$ be a semistable hyperelliptic curve defined by an integral Weierstraß equation $C:y^2=f(x)$ with cluster picture $\Sigma$.
	Let $\regX/ R$ be the minimal regular model. Assume that the residue field $k$ is algebraically closed. 
	
	Choose clusters $\cs_0, \dots, \cs_{g-1}$  inductively such that
	\[e_i = \frac{\nu_{\cs_i}}{2} - \sum_{j=0}^{i}d_{\cs_j \land \cs_i} = \max \left( \frac{\nu_{\cs}}{2} - \sum_{j=0}^{i-1}d_{\cs_j \land \cs} -d_{\cs}\right),\]
	where the maximum is taken over all proper clusters in $\Sigma$.
	If the maximal value is obtained by two different clusters $\cs$ and $\cs'$ with $\cs' \subset \cs$, choose  $\cs_i = \cs$. 
	
	Then an $R$-basis for the global sections of the relative dualising sheaf $\omega_{\regX/R}$ is given by $(\mu_0, \dots \mu_{g-1})$, where
	\[\mu_i = \pi^{e_i} \prod_{j=0}^{i-1}(x-z_{\cs_{j}}) \frac{dx}{2y}.\]
	
\end{thm}

Note that the sequence constructed in the theorem is not canonical, since it can happen that the maximum is obtained by two incomparable clusters, in which case any of the two clusters may be chosen. Before we illustrate the construction by an example, we make an observation concerning the assumptions in the theorem.

\begin{rem}
	\label{rem:k_closed}
	In Theorem \ref{thm:mainbasis_int}, it is assumed that the residue field is algebraically closed. This assures that every proper cluster has a $K$-rational centre. However, it should be pointed out that the construction of the minimal regular model commutes with taking the maximal unramified extension. 
	More precisely, let $K$ be a local field with ring of integers $R$, perfect residue field $k$  and $K^{ur}$ the maximal unramified extension with ring of integers $R'$. Let $C$ be a smooth projective curve over $K$ and $\regX$ its minimal regular model. Then $\regX'= \regX \times_{\Spec R} \Spec R'$ is the minimal regular model of $C$ over $R'$. This follows from Lemma 10.3.33.(b) in \cite{liu2002alggeo}.
	
	As a consequence $H^0(\regX',\omega_{\regX'/R'}) = H^0(\regX,\omega_{\regX/R}) \otimes_{R} R'$.
	Moreover, if we find a basis for $H^0(\regX',\omega_{\regX'/R'})$ which is defined over $K$, it is also a basis for $H^0(\regX,\omega_{\regX/R})$. 
	Translated to the situation of the theorem, this means that we can replace the assumption $k$ algebraically closed with the condition that every proper cluster admits a $K$-rational centre. 	
	
	 Note that for the same reason, the residue field $k$ does not need to be algebraically closed in the statement of  Theorem \ref{thm:main_lambda_int}.
\end{rem}

\begin{ex}
	\label{ex:basis}
	Let $p>3$ and $C/\QQ_p$ the hyperelliptic curve of genus $g = 5$ defined by
	\[y^2 = x(x-p^6)(x-2p^6)(x-p^4)(x-2p^4)(x-3p^4)(x-1)(x-1-p^8)(x-1-2p^8)(x-3p^8)(x-2)(x-3).\] The proper clusters are
	\[ \cR,\; \ct_1 = \{0,p^6, 2p^6, p^4,2p^4, 3p^4\},\; \ct_2 = \{0, p^6, 2p^6\},\; \ct_3 = \{1, 1+p^8, 1+2p^8,1+3p^8\}. \]
	These clusters have depths $d_\cR = 0, \; d_{\ct_1} = 4, \; d_{\ct_2} = 6$, $d_{\ct_3} = 8$ and relative depths $\delta_{\ct_1} = 4, \; \delta_{\ct_2} = 2$, $\delta_{\ct_3} = 8$. This information is contained in the cluster picture $\Sigma$:
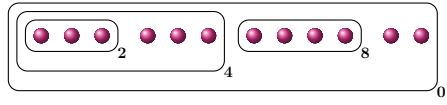
	\begin{center}
		\clusterpicture
		\Root(0.00,2)A1(a1);     
		\Root(0.40,2)A1(a2);     
		\Root(.80,2)A1(a3);      
		\Root(1.4,2)A1(b1);
		\Root(1.8,2)A1(b2);
		\Root(2.2,2)A1(b3);
		\Root(2.8,2)A1(c1);
		\Root(3.2,2)A1(c2);
		\Root(3.6,2)A1(c3);
		\Root(4.0,2)A1(c4);
		\Root(4.6,2)A1(d1);
		\Root(5.0,2)A1(d2);
		\ClusterL(C3){(a1)(a2)(a3)}{$\mathbf{2}$};   
		\ClusterL(C2){(C3)(C3n)(b1)(b2)(b3)}{$\mathbf{4}$};   
		\ClusterL(C4){(c1)(c2)(c3)(c4)}{$\mathbf{8}$};
		\ClusterL(C1){(C2)(C2n)(C4)(C4n)(d1)(d2)}{$\mathbf{0}$};
		\endclusterpicture
	\end{center}
	The subscript of the top cluster is its depth. The subscripts of the other clusters are their relative depths. 
	
	We construct a sequence of clusters as described in Theorem \ref{thm:mainbasis_int}. The theorem is applicable since every cluster has a centre in $\QQ_p$ (see the previous remark).
	First we choose $\cs_0$ to be the cluster that maximises $\frac{\nu_{\ct}}{2}-d_{\ct}$. The evaluation of this term for each cluster can be found in the first column (after the double line) of the table below. Next we choose $\cs_1$ to be the cluster that maximises $\frac{\nu_{\ct}}{2}-d_{\ct}- d_{\ct\land\cs_0}$, the cluster $\cs_2$ is the one that maximises $\frac{\nu_{\ct}}{2}-d_{\ct}- d_{\ct\land\cs_0}- d_{\ct \land\cs_1}$ and so on. 
	
	\begin{center}
		\begin{tabular}{ l | c | c || c | c | c | c | c }
			&$\nu_{\ct}$&$d_{\ct}$&$\frac{\nu_{\ct}}{2}-d_{\ct}$&$\frac{\nu_{\ct}}{2}-d_{\ct} - d_{\ct \land \ct_2}$&\dots $- d_{\ct \land \ct_3}$&\dots$-d_{\ct \land \ct_1}$&$\dots - d_{\ct \land \cR}$\\ \hline
			$\cR$   & 0  & 0 & 0  & 0 & 0 & \numcirclemodi{0}&\numcirclemodi{0}\\ \hline
			$\ct_1$ & 24 & 4 & 8 & 4& \numcirclemodi{4}& 0 &0\\ \hline
			$\ct_2$ & 30 & 6 & \numcirclemodi{9} & 3&3& -1 &-1\\ \hline
			$\ct_3$ & 32 & 8 & 8 &  \numcirclemodi{8}&0& 0 &0\\
		\end{tabular}
	\end{center}
	In each column the maximal value is circled to indicate which cluster is chosen in the respective step. Three dots always represent the entire expression in the previous column. We can read off from the table
	
	\begin{center}
		\bgroup
		\def\arraystretch{1.5}
		\begin{tabular}{ c  c  c  c  c c }	
			&$\cs_0 = \ct_2$,  & $\cs_1 = \ct_3$, & $\cs_2 = \ct_1$,&$\cs_3 = \cR$,&$\cs_4 = \cR$, \\
			and &$e_0 = 9,$ &$e_1 = 8$, & $e_2 = 4$, &$e_3 = 0$, &$e_4 = 0$.\\
		\end{tabular}
		\egroup
	\end{center}
	If we choose $z_{\cR} = z_{\ct_1} = z_{\ct_2} = 0 \textrm{ and } z_{\ct_3} = 1$ as centres for the clusters, 
	we get the following basis for the global sections of $\omega_{\regX/R}$:
	\[ \left( \mu_0 = p^9 \,\frac{dx}{2y}, \;\mu_1 = p^8 x\,\frac{dx}{2y}, \; \mu_2 = p^4 x(x-1)\, \frac{dx}{2y}, \; \mu_3 = x^2(x-1) \,\frac{dx}{2y}, \; \mu_4 = x^3(x-1)\, \frac{dx}{2y} \right)\] 
\end{ex}

As mentioned before, it is often not necessary to determine the basis for $H^0(\regX,\omega_{\regX/R})$ explicitly, but it suffices to know a generator for $\det H^0(\regX,\omega_{\regX/R})$. Theorem \ref{thm:main_lambda_int} gives a convenient formula for determining this generator.

\begin{thm}[Theorem \ref{thm:main_lambda}]
	\label{thm:main_lambda_int}
	Let $C/K$ be a semistable hyperelliptic curve defined by $C: y^2=f(x)$ with $f(x) = c_f\prod_{r\in \cR}(x-r)$. We write $\omega_0 = \frac{dx}{2y}, \dots, \omega_{g-1} = \frac{x^{g-1}\,dx}{2y}$ for the differentials associated to this equation and $\omega := \omega_0 \land \dots \land \omega_{g-1} \in  \det H^0(C,\Omega^1_{C/K})$. Let $\regX \rightarrow \Spec R$ be the minimal regular model of $C$. Suppose that $\lambda_C \cdot \omega$ is a basis for $\det H^0(\regX,\omega_{\regX/R})$. Then  
	\begin{align*}
	\begin{split}
	8 \, v(\lambda_C) =  &\; 4\, g \cdot v(c_f) + \sum_{\substack{|\cs| \textrm{ even}\\ \cs \neq \cR}} \delta_{\cs}  (|\cs|-2)|\cs| + \sum_{\substack{|\cs| \textrm { odd}\\ \cs \neq \cR}} \delta_{\cs} (|\cs|-1)^2\\
	&~~~~+ d_{\cR} 
	\begin{cases} 
	(|\cR|-2)|\cR|, & \text{if $|\cR| = 2g+2$}\\
	(|\cR|-1)^2, & \text{if $|\cR| = 2g+1$}
	\end{cases}.
	\end{split}
	\end{align*} 	
\end{thm}

Let us revisit the above example.

\begin{ex}
	\label{ex:lambda}
	Let $p>3$ and $C/\QQ_p^{ur}$ be the hyperelliptic curve from Example \ref{ex:basis}. Recall the cluster picture $\Sigma$:
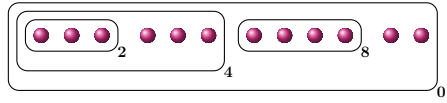
	\begin{center}
		\clusterpicture
		\Root(0.00,2)A1(a1);     
		\Root(0.40,2)A1(a2);     
		\Root(.80,2)A1(a3);      
		\Root(1.4,2)A1(b1);
		\Root(1.8,2)A1(b2);
		\Root(2.2,2)A1(b3);
		\Root(2.8,2)A1(c1);
		\Root(3.2,2)A1(c2);
		\Root(3.6,2)A1(c3);
		\Root(4.0,2)A1(c4);
		\Root(4.6,2)A1(d1);
		\Root(5.0,2)A1(d2);
		\ClusterL(C3){(a1)(a2)(a3)}{$\mathbf{2}$};   
		\ClusterL(C2){(C3)(C3n)(b1)(b2)(b3)}{$\mathbf{4}$};   
		\ClusterL(C4){(c1)(c2)(c3)(c4)}{$\mathbf{8}$};
		\ClusterL(C1){(C2)(C2n)(C4)(C4n)(d1)(d2)}{$\mathbf{0}$};
		\endclusterpicture
	\end{center}
	Then the formula of Theorem \ref{thm:main_lambda_int} yields
	\begin{align*}
	8 \, v(\lambda_C) &=   4 \cdot g \cdot v(c_f) &+\;& \delta_{\ct_1}(|\ct_1|-2)|\ct| &&+\; \delta_{\ct_2}(|\ct_2|-1)^2 &+\;& \delta_{\ct_3}(|\ct_3|-2) |\ct_3| &+\;& d_{\cR}\, (|\cR|-2)|\cR|\\
	&= 4 \cdot 5 \cdot 0 &+\;& 4 \cdot 4 \cdot 6 &&+\; 2 \cdot 2^2 &+\;& 8 \cdot 2 \cdot 4 &+\;& 0 \cdot 10 \cdot 12\\
	&= 8 \cdot 21. &&&&&&&&
	\end{align*}
	So \[\mu: = p^{21} \, \frac{dx}{2y} \land \frac{xdx}{2y}\land \frac{x^2dx}{2y}\land \frac{x^3dx}{2y}  \land \frac{x^4dx}{2y} \]
	generates $\det H^0(\regX,\omega_{\regX/R})$. Note that the value $v(\lambda_C) = 21$ is equal to the sum over the $e_i$ determined by Theorem \ref{thm:mainbasis_int}.
\end{ex}

\subsection{Outline}
In Section \ref{sec:cluster}, we review some definitions and facts about cluster pictures. In Section \ref{sec:lambda} we translate Proposition 5.5. of \cite{kausz1999discriminant} into the language of cluster pictures and generalise it in order to prove Theorem \ref{thm:main_lambda_int}. The last section is dedicated to the proof of Theorem \ref{thm:mainbasis_int}. This is done by first showing that the differentials forms defined in the theorem are indeed global sections of the canonical sheaf and then applying Theorem \ref{thm:main_lambda_int}.

\subsection*{Acknowledgements}
I would like to thank Vladimir Dokchitser for proposing to work on this topic and his support throughout the creation of this paper. I would also like to thank Stefan Wewers for very helpful discussions and comments on earlier versions of this paper, as well as Adam Morgan who also suggested a proof of Proposition 3.3.

\section{Cluster Pictures}
\label{sec:cluster}
In this section, we describe the cluster picture associated to an equation defining a hyperelliptic curve and briefly introduce the notation used in the subsequent sections. All information is taken from \cite{dokchitser2018arithmetic}.

Let $C/K$ be a hyperelliptic curve defined by a Weierstraß equation $C:y^2 = f(x)$. We write $\cR$ for the set of roots of $f(x)$ in $K^{sep}$ and $c_f$ for its leading coefficient, so that
\[f(x)= c_f\prod_{r\in \cR}(x-r).\] 

\begin{defs}
	\begin{enumerate}[(i)]
		\item (\cite{dokchitser2018arithmetic} Definition 1.1.) A \textit{cluster} is a non-empty subset $\cs \subset \cR$ of the form $\cs = D\cap \cR$ for some disc $D = \{x \in \bar{K} \,|\,v(x-z) \geq d\}$ for some $d\in \QQ$ and $z\in \bar{K}$. We say that $z$ is a \textit{centre} of the cluster and write $z = z_{\cs}$.
		\item 	(\cite{dokchitser2018arithmetic} Definition 1.1.) If $|\cs|>1$, then $\cs$ is called a \textit{proper} cluster and its \textit{depth} is defined to be $d_{\cs} = \min_{r,r' \in \cs} v(r-r')$.
		\item (\cite{dokchitser2018arithmetic} Definition 1.3.) If $\cs' \subsetneq \cs$ is a maximal subcluster, we write $\cs' < \cs$ and say that $\cs'$ is a child of $\cs$.
		For two clusters (possibly roots) $\cs,\; \cs'$ we write $\cs \land \cs'$ for the smallest cluster that contains them.
		\item (\cite{dokchitser2018arithmetic} Definition 1.4.) A cluster $\cs$ is \textit{principal} if $|\cs| \geq 3$,
		except if either $\cs = \cR$ is even and has exactly two children, or if $\cs$ has a child of size $2g$.
		\item (\cite{dokchitser2018arithmetic} Definition 1.5.) If $\cs \neq \cR$ is proper, the \textit{relative depth} of  $\cs$ is defined as the difference between the depth of $\cs$ and the depth of the smallest cluster strictly containing $\cs$.  It is denoted by $\delta_{\cs}$.	
		\item (\cite{dokchitser2018arithmetic} Definition 1.6.) For a proper cluster $\cs$, set $\nu_{\cs} = v(c_f)+\sum_{r \in \cR} d_{r \land \cs}$.
	\end{enumerate} 
\end{defs}

By the \textit{cluster picture} associated to an equation, we mean the collection of clusters $\{\cs \subset \cR\}$ together with their depths. See Example \ref{ex:basis} in the introduction for an illustration of these definitions.

\begin{rem}
	\label{rem:compute_vs}
	Note that for a root $r \in \cR$ and a cluster $\cs$, we have \[d_{r \land \cs} = d_{\cR} + \sum_{\substack{\cs' \neq \cR:\\ \cs \land r \subseteq \cs'  }} \delta_{\cs'}.\] So $\nu_{\cs}$ can also be calculated via \[\nu_{\cs} = v(c_f)+d_{\cR}|\cR| +  \sum_{\substack{\cs' \neq \cR:\\ \cs \subseteq \cs'  }} \delta_{\cs'} |\cs'|. \]
\end{rem}

There is a notion of equivalence for cluster pictures that respects isomorphisms of hyperelliptic curves. For a complete discussion of the topic we refer to  \cite[Section 14]{dokchitser2018arithmetic}. The following proposition is important for the proof of Theorem \ref{thm:main_lambda}. So we state it here for the convenience of the reader.

\begin{prop}[\cite{dokchitser2018arithmetic}, Proposition 14.6.]
	\label{prop:manipulateclusters}
	Let $f(x) \in K[x]$ be a separable polynomial with roots $\cR \subset K^{sep}$, such
	that the absolute Galois group $G_K = G(K^{sep}/K)$ acts tamely on $\cR$, and let $\Sigma$ be the associated cluster picture. Suppose $\Sigma'$ is a cluster picture obtained from $\Sigma$ by one of the following constructions:
	\begin{enumerate}
		\item Increasing the depth of all clusters by some $n\in \ZZ$;
		\item Adding a root to $\cR$, provided $|\cR|$ is odd, $d_{\cR} \in \ZZ$ and $|k|> \#\{\cs < \cR : \cs \textrm{ is } G_K-\textrm{stable}\}$;
		\item Redistributing the depth between $\cs$ and $\cR \backslash \cs$ by decreasing the depth of $\cs$ by $1$ and increasing the depth of $\cR \backslash \cs$ by $1$,
		provided $|\cR|$ is even, $\cs < \cR$ is $G_K$-stable with $d_{\cR}, \,d_{\cs} \in \ZZ$ and $|k| > \#\{\ct < \cs : \ct \textrm{ is } G_K- \textrm{stable} \}$.
		\footnote{If $\cR \backslash \cs$ is not contained in $\Sigma$, it is added to the cluster picture as a new cluster with relative depth $\delta_{\cR\backslash\cs} =0$. In accordance with \cite[Definition 14.1.]{dokchitser2018arithmetic} redistributing the depth between $\cs$ and $\cR \backslash \cs$ includes decreasing (respectively increasing) the depth of all proper clusters contained in $\cs$ (respectively $\cR \backslash \cs$).}
	\end{enumerate}
	Then there is a Möbius transformation $\phi(z) = \frac{az+b}{cz+d}$ with $a, b, c, d \in K$, such that $\Sigma'$	is the cluster picture of $\cR' = \{\phi(r): r \in \cR \} \backslash \{\infty\}$ if $|\cR|$ is even and of $\cR' = \{\phi(r): r \in \cR \cup \{\infty\}\}\backslash \{\infty\}$ if $|\cR|$ is odd.	Moreover, if $y^2 = f(x)$ is a hyperelliptic curve, then there is a $K$-isomorphic curve given by a Weierstraß model whose cluster picture is $\Sigma'$.
\end{prop}

\section{A Basis for $\det H^0(\regX,\omega_{\regX/R})$}
\label{sec:lambda}

Let $C/K$ be a semistable hyperelliptic curve of genus $g$ defined by $C: y^2=f(x)$ with $f(x) = c_f\prod_{r\in \cR}(x-r)$. We write $\omega_0 = \frac{dx}{2y}, \dots, \omega_{g-1} = \frac{x^{g-1}\,dx}{2y}$ for the differentials associated to this equation and $\omega := \omega_0 \land \dots \land \omega_{g-1} \in  \det H^0(C,\Omega^1_{C/K})$. Let $\regX \rightarrow \Spec R$ be the minimal regular model of $C$. The main result of this section is Theorem \ref{thm:main_lambda}, where we determine $\lambda_C \in K$, such that $\lambda_C \cdot \omega$ generates  $\det H^0(\regX,\omega_{\regX/R})$ as an $R$-module. Note that $\lambda_C$ is only well-defined up to a unit. Moreover it is not a curve invariant, but depends on the equation.

\begin{thm}
	\label{thm:main_lambda}
	Let $C/K$ be a semistable hyperelliptic curve of genus $g$ defined by $C: y^2=f(x)$ with $f(x) = c_f\prod_{r\in \cR}(x-r)$. We write $\omega_0 = \frac{dx}{2y}, \dots, \omega_{g-1} = \frac{x^{g-1}\,dx}{2y}$ for the differentials associated to this equation and $\omega := \omega_0 \land \dots \land \omega_{g-1} \in  \det H^0(C,\Omega^1_{C/K})$. Let $\regX \rightarrow \Spec R$ be the minimal regular model of $C$. Suppose that $\lambda_C \cdot \omega$ is a basis for $\det H^0(\regX,\omega_{\regX/R})$. Then  
	\begin{align}
	\label{eqn:lambda}
	\begin{split}
	8 \, v(\lambda_C) =  &\; 4\, g \cdot v(c_f) + \sum_{\substack{|\cs| \textrm{ even}\\ \cs \neq \cR}} \delta_{\cs}  (|\cs|-2)|\cs| + \sum_{\substack{|\cs| \textrm { odd}\\ \cs \neq \cR}} \delta_{\cs} (|\cs|-1)^2\\
	&~~~~+ d_{\cR} 
	\begin{cases} 
	(|\cR|-2)|\cR|, & \text{if $|\cR| = 2g+2$}\\
	(|\cR|-1)^2, & \text{if $|\cR| = 2g+1$}
	\end{cases}.
	\end{split}
	\end{align} 	
\end{thm}

A result in \cite{kausz1999discriminant} shows that this formula is true under several additional assumptions (Lemma \ref{lem:kausz}). Our strategy for the proof of the theorem is to pass to a tamely ramified extension of the base field and apply several Möbius transformations until Lemma \ref{lem:kausz} can be applied. In the course of these transformations the valuation of $\lambda$ changes. Hence it is necessary to add correction terms to the original formula. Their computation can be reduced to computing the change of the discriminant of the equation under these Möbius transformations, which has already been done in \cite[Section 16]{dokchitser2018arithmetic}.  A key ingredient for this reduction step is the \textit{hyperelliptic discriminant}, see Definition \ref{def:hyperellipticDiscriminant}. This is a curve invariant which connects the valuation of $\lambda$ with the valuation of the discriminant of a Weierstraß equation.

\begin{rem}
	\label{rem:notation_kausz}
	The starting point for our proof is \cite[Proposition 5.5.2]{kausz1999discriminant}. The results in that article are phrased in a different language. For the convenience of the reader, we shortly introduce the necessary notation and explain how it compares to cluster pictures. This notation will only be used in the proof of the following lemma.
	
	Throughout, let $y^2=f(x)$ be an equation defining a semistable hyperelliptic curve. Assume that the set of roots $\cR$ of $f$ is contained in $R$. 
	Hence for any $n \in \NN \cup \{0\}$, there is a natural map
	\[
	\rho_n: \cR \to R/\cp^n
	\] and one may define
	\[
	\mathcal{V}_n = \{V \in R/\cp^n \mid \#\rho_n^{-1}(V) \geq 2\}.
	\]
	The set
	\[
	\mathcal{V}(T) = \cup_{n\geq 0} \mathcal{V}_n
	\]
	naturally has the structure of a rooted tree. 
	For each $V$ in $\mathcal{V}(T)$, one defines $\phi(V) = |\rho_n^{-1}(V)|$ and $\gamma(V) = \lfloor (\phi(V) - 1)/2 \rfloor$. 
	
	Let us now compare the tree $\mathcal{V}(T)$ to the cluster picture of the equation. Clearly, $\cR$ corresponds to the vertex $V_0$. Let $\cs = \{r_1, \dots, r_m\} \subsetneq \cR$ be a proper cluster with depth $d_{\cs}$ and relative depth $\delta_{\cs}$. Then $\rho_n(\cs) \in \mathcal{V}_n$ if and only if $d_{\cs} -\delta_{\cs}< n \leq d_{\cs}$. Roughly speaking, for each proper cluster $\cs \subsetneq \cR$, the tree $\mathcal{V}(T)$ contains $\delta_{\cs}$ copies of this cluster. Note that $\delta_{\cs}$ is integral by assumption. On the other hand, it is easy to see that each vertex in $\mathcal{V}(T)$ corresponds to some proper cluster and the value $\phi(V)$ is the cardinality of that cluster.
\end{rem}

\begin{lem}
	\label{lem:kausz}
	Let $C/K$ be a hyperelliptic curve defined by $C: y^2=f(x)$ with $f(x) = c_f\prod_{r\in \cR}(x-r)$. Assume that 
\begin{multicols}{2}
	\begin{enumerate}[(i)]
		\item $\cR \subset R$,
		\item $v(r-s) \in 2\ZZ$ for all $ r,\, s \in \cR$,
		\item $c_f$ is a unit in $R$,
		\item $\#\cR = 2g+2$,
		\item $\#\{\bar{r} \,|\, r\in \cR\} \geq 3$.
		\item[]
	\end{enumerate}
\end{multicols}
	Then $v(\lambda_C)$ can be computed using Equation \ref{eqn:lambda}.
\end{lem}

\begin{proof}
	Under the conditions in the lemma, Equation \ref{eqn:lambda} reduces to 
	\begin{align}
	\label{eqn:simple}
	8 \,v(\lambda_C) = \sum_{\substack{|\cs| \textrm{ even}\\ \cs \neq \cR}} \delta_{\cs} (|\cs|-2)|\cs| + \sum_{\substack{|\cs| \textrm { odd}\\ \cs \neq \cR}} \delta_{\cs} (|\cs|-1)^2.
	\end{align}
	This formula is due to Kausz, see \cite[Proposition 5.5.2.]{kausz1999discriminant}. The formula in the original source reads as
	\[
	\sum_{i=0}^{g-1} e_i = \frac{1}{2} \sum_{\substack{V>V_0\\\phi(V) \textrm{ even}}} \gamma(V)(\gamma(V) + 1) + \frac{1}{2} \sum_{\substack{V>V_0\\\phi(V) \textrm{ odd}}} \gamma(V)^2
	\] with the notation introduced in the remark preceding this lemma.  Note that the $e_i$'s play the same role as the $e_i$'s defined in Theorem \ref{thm:mainbasis} and it directly follows from the first part of the proposition in \cite{kausz1999discriminant} that $\sum e_i = v(\lambda_C)$. Phrased in terms of cluster pictures, this is precisely  Formula \ref{eqn:simple}. Several conditions are imposed on the equation defining $C$ (cf. \cite[Lemma 4.1]{kausz1999discriminant}). It is easy to see that these are implied by Conditions (i)-(v).
	Moreover, it is assumed that the residue field $k$ is algebraically closed, but this has no effect on the valuation of $\lambda_C$, see Remark \ref{rem:k_closed}.
\end{proof}

Note that the conditions in the above lemma imply semistability, see for example \cite[Theorem 7.1.]{dokchitser2018arithmetic}. Conversely, after a tamely ramified extension of the base field, there always exists an equation for $C$ satisfying the conditions listed in the above lemma if $C$ is semistable.

\begin{prop}
	\label{prop:H0basechange}
	Let $C/K$ be a semistable hyperelliptic curve with minimal regular model $\regX \rightarrow \Spec R$. Let $K'/K$ be a finite field extension. Write $C'$ for the base-change of $C$ to $K'$, $R' = \mathcal{O}_{K'}$ and $\regX'\rightarrow \Spec R'$ for the minimal regular model of $C'$.
	Then
	\[H^0(\regX', \omega_{\regX'/R'}) = H^0(\regX, \omega_{\regX/ R}) \otimes_R R'\]
	inside $H^0(C', \Omega_{C'/K'}) = H^0(C, \Omega_{C/K}) \otimes_K K'$. 
\end{prop}

\begin{proof}
	Let $\stY \rightarrow \Spec R$ and $\stY' \rightarrow \Spec R'$ be the stable models of $C/K$ and $C'/K'$ respectively.
	The stable model $\stY$ is obtained from $\regX$ by contraction of all components $\Gamma$ of the special fibre for which $K_{\regX/ R} . \Gamma = 0$. Write $f: \regX \rightarrow \stY$ for the contraction morphism. Since the intersection matrix of the contracted components is negative definite, it follows from \cite[Corollary 9.4.18.]{liu2002alggeo} that $\omega_{\regX/R} = f_*\omega_{\stY/R}$. Therefore $ H^0(\regX, \omega_{\regX/R}) = H^0(\stY,\omega_{\stY/R})$ inside $H^0(C, \Omega_{C/K})$. By the same reasoning $ H^0(\regX', \omega_{\regX'/ R'}) = H^0(\stY',\omega_{\stY'/ R'})$.
	So it suffices to show that $H^0(\stY',\omega_{\stY'/R'}) = H^0(\stY,\omega_{\stY/R}) \otimes_R R'$.
	
	 We have $\stY' = \stY \times_R R'$ and by \cite[Theorem 6.4.9.b]{liu2002alggeo} $\omega_{\stY'/ R'} = p^* \omega_{\stY/R}$, where $p : \stY \times_R R' \rightarrow \stY$ is the first projection. So the result follows from \cite[Corollary 5.2.27]{liu2002alggeo}.
\end{proof}

\begin{lem}
	\label{lem:tame}
	Let $C/K$ be a hyperelliptic curve with semistable reduction defined by $C: y^2=f(x)$. Let $K'/K$ be a finite extension and write $C'$ for the base-change of $C$ to $K'$. Then Equation \ref{eqn:lambda} holds for $C/K$ if and only if it holds for $C'/K'$.
\end{lem}

\begin{proof}
	 Let $e_{K'/K}$ denote the ramification degree of the extension $K'/K$. We write $\Sigma$ for the cluster picture associated to $C:y^2 = f(x)$ and $\Sigma'$ for the cluster picture associated to the equation $C':y^2=f(x)$ over $K'$. The clusters themselves do not change under a finite extension, but their depths do. More precisely we have $\delta_{\cs}' = e_{K'/K} \cdot \delta_{\cs}$ for each proper cluster $\cs\neq \cR$  and $d_{\cR}' = e_{K'/K} \cdot d_{\cR}$.
	 We write $R' := \mathcal{O}_{K'}$ and $v'$ for the normalised valuation. That is $v'(r) = e_{K'/K}\cdot v(r)$ for all $r\in R$.
	 
	 For $\lambda_{C'}$, Equation \ref{eqn:lambda} yields
	 \begin{align*}
	 8\,v'(\lambda_{C'})
	 = &e_{K'/K} \cdot \Bigg(g \cdot v(c_f) +  \sum_{\substack{|\cs| \textrm{ even}\\ \cs \neq \cR}}  \delta_{\cs} (|\cs|-2)|\cs| + \sum_{\substack{|\cs| \textrm { odd}\\ \cs \neq \cR}}  \delta_{\cs}(|\cs|-1)^2 \\
	 &~~~~~~~~~+ d_{\cR} \begin{cases} (|\cR|-2)|\cR|, & \text{if $|\cR| = 2g+2$}\\
	(|\cR|-1)^2, & \text{if $|\cR| = 2g+1$} 
	 \end{cases}\Bigg).\\
	 \end{align*}
	 From Proposition \ref{prop:H0basechange} it follows that $v'(\lambda_{C'})=e_{K'/K} \cdot v(\lambda_C)$. So the above calculation shows that Formula \ref{eqn:lambda} is true for $C/K$ if and only if it is true for $C'/K'$.
\end{proof}

\begin{defn}
	Let $C/K$ be a hyperelliptic curve of genus $g$, defined by some Weierstraß equation $y^2 =f(x)$. We denote by $c_f$ the leading coefficient of $f$. Then the \textit{discriminant} $\Delta$ of the equation is defined as \[\Delta :=  2^{4g}\, c_f^{4g+2} \disc \left(\frac{1}{c_f} \,f(x)\right).\]
\end{defn}

While the discriminant defined above is not a curve invariant, but depends on the equation, there exists a more natural definition of discriminant. See also the paragraph before Proposition 2.2. in \cite{kausz1999discriminant}.

\begin{defn}
	\label{def:hyperellipticDiscriminant}
	Let $C/K$ be a hyperelliptic curve of genus $g$, defined by some Weierstraß equation $y^2 =f(x)$ with discriminant $\Delta$. We associate to this equation the differential forms $\omega_0, \dots, \omega_{g-1}$ and write $\omega = \omega_0 \land \dots \land \omega_{g-1} \in H^0(C,\Omega^1_{C/K})$.
	
	Then the element
	\[\Lambda:=\Delta^g \cdot \omega^{\otimes 8g+4} \in (\det H^0(C,\Omega^1_{C/K}))^{\otimes 8g+4}\]
	is called \textit{hyperelliptic discriminant} of $C$.
\end{defn}

The following proposition shows that $\Lambda$ is well-defined.

\begin{prop}
	\label{fct:disc}
	Let $C/K$ be a hyperelliptic curve with hyperelliptic discriminant $\Lambda$. Let $y^2 =f(x)$ be some Weierstraß equation for $C$. We associate to this equation the elements $\Delta$ and $\lambda$.
	Then the following statements are true.
	\begin{enumerate}
		 \item The element $\Lambda$ is independent of the choice of equation. 
		 \item  Viewed as a rational section of $(\det H^0(\regX,\omega_{\regX/R}))^{\otimes 8g+4}$, the order of vanishing in $\mathfrak{p}$ is given by
		\[\ord_{\mathfrak{p}}(\Lambda) = g\cdot v(\Delta) - (8g+4) \cdot v(\lambda_C). \]
		\item Let $y'^2 = g(x')$ be another equation defining the same curve with $\Delta'$ and $\lambda_{C'}$ the corresponding quantities. Then
	\[\frac{v(\Delta') - v(\Delta)}{8g+4} = \frac{v(\lambda_{C'})-v(\lambda_C)}{g}.\]
	\end{enumerate}
\end{prop}

\begin{proof}
	\begin{enumerate}
		\item This is Proposition 2.2.1. in \cite{kausz1999discriminant}.
		\item This follows from the definition of $\Lambda$. See also \cite[Section 5, Formula 1]{kausz1999discriminant}.
		\item This is a direct consequence of the first two statements.
	\end{enumerate}
\end{proof}
	
\begin{lem}
	\label{lem:cf}
	Let $C/K$ be a hyperelliptic curve defined by $C:y^2 =f(x)$ with $f(x)=c_f\prod_{r\in \cR}(x-r)$ and $v(c_f)$ in $2\ZZ$. Then the curve $C'/K$ defined by  $y'^2 = \prod_{r\in \cR}(x-r)$ is isomorphic to $C$ over $K^{nr}$ and
	\[v(\lambda_C) = v(\lambda_{C'}) +g \cdot \frac{v(c_f)}{2}.\]
	
\end{lem}

\begin{proof}
	Since $v(c_f)$ is even, $c_f$ is a square in $K^{nr}$ and the two equations define isomorphic curves over $K^{nr}$.
	
	By definition, the discriminant of the equation $y^2=c_f\prod_{r\in \cR}(x-r)$ is \[\Delta = 2^{4g}\, c_f^{4g+2}\,\disc(\prod_{r\in \cR}(x-r)).\] So 
	\[v(\Delta) = v(\Delta') + (4g+2)\, v(c_f).\]
	Using part 3 of Proposition \ref{fct:disc}, we get
	\[v(\lambda_C) = v(\lambda_{C'}) +g \cdot \frac{v(c_f)}{2}. \]
\end{proof}

\begin{lem}
	\label{lem:manipulatesigma}
	Let $C/K$ be a hyperelliptic curve and $y^2 = f(x)$ a Weierstraß equation defining this curve. Let $\Sigma$ be its cluster picture and $\lambda_C$ the quantity associated to this equation.
	
	Let $y'^2=g(x')$ be a different equation for $C$. Denote by $\Sigma'$ and $\lambda_{C'}$ the corresponding elements.
	\begin{enumerate}
		\item \label{lem:increasedepth} If $\Sigma'$ is obtained from $\Sigma$ by increasing the depths of all clusters by some $t\in \ZZ$, then
		\[v(\lambda_C) = v(\lambda_{C'}) - \frac{t}{8} \cdot \begin{cases} (|\cR|-2)|\cR|, & \text{if $|\cR| = 2g+2$}\\
		(|\cR|-1)^2, & \text{if $|\cR| = 2g+1$}
		\end{cases}\\ \]
		\item \label{lem:addroot}If $\Sigma'$ is obtained from $\Sigma$ by adding a root to $\cR$, then
		\[v(\lambda_C) = v(\lambda_{C'}) - \frac{d_{\cR} (|\cR|-1)}{4}.\]
		\item \label{lem:redistributedepth} If $\cR$ has even size and $\Sigma'$ is obtained from $\Sigma$ by redistributing the depth between $\cs <\cR$ and $\cR \backslash \cs$ to $d_{\cs}' = d_{\cs} - t$ and  $d_{\cR \backslash \cs}' = d_{\cR \backslash \cs} +t$, then 
		\[v(\lambda_C) = v(\lambda_{C'}) - t \cdot \frac{(|\cR|-2)(|\cR|-2|\cs|)}{8}.\]
	\end{enumerate}
\end{lem}

\begin{proof}
	From \cite[Lemma 16.6.]{dokchitser2018arithmetic}, we know how the discriminant changes under the above modifications of the cluster picture. We will combine these results with Proposition \ref{fct:disc}, part 3.
	\begin{enumerate}
		\item By  \cite[Lemma 16.6.(i)]{dokchitser2018arithmetic}
		\[v(\Delta') - v(\Delta) = t |\cR|(|\cR|-1).\]
		Now Proposition \ref{fct:disc} implies
		\begin{align*}
		v(\lambda_{C'})-v(\lambda_C) = &\,g\cdot \frac{ t\, |\cR|(|\cR|-1)}{8g+4}\\
		=&\,\frac{t}{8} \cdot \begin{cases} (|\cR|-2)|\cR|, & \text{if $|\cR| = 2g+2$}\\
		(|\cR|-1)^2, & \text{if $|\cR| = 2g+1$}
		\end{cases}\\
		\end{align*}
		\item By \cite[Lemma 16.6.(ii)]{dokchitser2018arithmetic}
		\[v(\Delta') - v(\Delta) = 2\, d_{\cR}\, |\cR|.\]
		Now Proposition \ref{fct:disc} implies
		\begin{align*}
		v(\lambda_{C'})-v(\lambda_C) = &g\cdot \frac{2\,d_{\cR}\, |\cR|}{8g+4}\\
		=& \frac{|\cR|-1}{2} \cdot \frac{d_{\cR}}{2}.
		\end{align*}
		\item By \cite[Lemma 16.6.(iv)]{dokchitser2018arithmetic}
		\[v(\Delta') - v(\Delta) = t \cdot (|\cR|-2|\cs|)(|\cR|-1).\]
		Now Proposition \ref{fct:disc} implies
		\begin{align*}
		v(\lambda_{C'})-v(\lambda_C) = &g\cdot \frac{t (|\cR|-2|\cs|)(|\cR|-1)}{8g+4}\\
		=& t\cdot \frac{|\cR|-2}{2} \cdot \frac{|\cR|-2|\cs|}{4}.
		\end{align*}
	\end{enumerate}
\end{proof}

\begin{proof}[Proof of Theorem \ref{thm:main_lambda}]
	Let $C/K$ be a hyperelliptic curve with semistable reduction defined by $C: y^2=f(x)$ with $f(x) = c_f\prod_{r\in \cR}(x-r)$. In Lemma \ref{lem:tame} we have seen that it suffices to prove that the formula holds after a finite extension. So we may assume that $\cR \subset K$,  $v(r-s) \in 2\ZZ$ for all $r,s \in \cR$, and $v(c_f) \in 2 \ZZ$. Subtracting the correction term $4g \cdot v(c_f)$ from the right hand side, we may even assume that $v(c_f) = 0$. This follows from Lemma \ref{lem:cf}.
	
	Further we can perform a Möbius transformation such that the cluster picture corresponding to the new equation has outer depth $d_\cR=0$ (see Proposition \ref{prop:manipulateclusters}, part 1). By Lemma \ref{lem:manipulatesigma} this corresponds to subtracting $ d_{\cR}  (|\cR|-2) |\cR|$ if $|\cR| = 2g+2$ or $d_{\cR} (|\cR|-1)^2$ if $|\cR| = 2g+1$ from $8v(\lambda_C)$. Note that decreasing the absolute depths does not change any relative depths.
	
	In case that  $|\cR| =2g+1$, we can perform a Möbius transformation that corresponds to adding one root to the cluster picture (as in Proposition \ref{prop:manipulateclusters}, part 2). Since $d_\cR=0$, this does not change the valuation of $\lambda_C$.
	After these two steps we are left with proving the simplified formula that already appeared in Lemma \ref{lem:kausz}. That is
	\[
		8 \,v(\lambda_C) = \sum_{\substack{|\cs| \textrm{ even}\\ \cs \neq \cR}} \delta_{\cs} (|\cs|-2)|\cs| + \sum_{\substack{|\cs| \textrm { odd}\\ \cs \neq \cR}} \delta_{\cs} (|\cs|-1)^2.
	\]
	with Conditions (ii), (iii) and (iv) of the lemma being satisfied and $d_{\cR} =0$.
	
	Without loss of generality we may always assume $G_K$-stability for clusters because the formula for  $\lambda_C$ behaves well under tamely ramified extension (see Lemma \ref{lem:tame}). So we can apply Part 3 of Proposition \ref{prop:manipulateclusters}, that is redistribute depth between a ($G_K$-stable) cluster $\cs$ and $\cR\backslash\cs$. This allows us to manipulate the cluster picture such that $\cR$ has at least three children in the following way. Assume that $\cR$ has exactly two children, $\cs_1$ and $\cs_2$. Without loss of generality assume that $\cs_1$ is not a root, hence it has at least two children on its own, say $\ct_1$ and $\ct_2$. After redistributing depth between $\cs_1$ and $\cs_2$ in such a way that $d_{\cs_1} = 0$, $\cs_1$ is no longer a cluster but $\cR$ has at least three children given by $\ct_1, \ct_2,\cs_2$. Together with the fact that $d_{\cR}=0$, this implies Condition (v) of Lemma \ref{lem:kausz}. 
	
	Let us show that the formula behaves well when redistributing depths. Let $\cs^* < \cR$ be a cluster in $\Sigma$. Let $\Sigma'$ be the cluster picture obtained after redistributing depth between $\cs^*$ and $\cR \backslash \cs^*$. That is $d_{\cs^*}' = d_{\cs^*} - t$ and  $d_{\cR \backslash \cs^*}' = d_{\cR \backslash \cs^*} +t$ for some $t\in \ZZ$. We have already seen in Lemma \ref{lem:manipulatesigma}, part 3 that 
	\[ 8\,\left(v(\lambda_{C'}) - v(\lambda_C)\right) =t \cdot (|\cR|-2)(|\cR|-2|\cs|).\]
	The calculation below shows that this equals the change on the right hand side of the equation.
\begin{align*}
\sum_{\substack{|\cs| \textrm{ even}\\ \cs \neq \cR}}& \delta_{\cs}' (|\cs|-2)|\cs| + \sum_{\substack{|\cs| \textrm { odd}\\ \cs \neq \cR}} \delta_{\cs}'(|\cs|-1)^2\\
=& \sum_{\substack{|\cs| \textrm{ even}\\ \cs \neq \cs^*,\cR \backslash \cs^*}} \delta_{\cs} (|\cs|-2)|\cs| +  \sum_{\substack{|\cs| \textrm { odd} \\ \cs \neq \cs^*,\cR \backslash \cs^*}} \delta_{\cs}(|\cs|-1)^2 \\
&~~~~~+ \begin{cases} (\delta_{\cs^*}-t) (|\cs^*|-2) |\cs^*| + (\delta_{\cR\backslash\cs^*}+t) (|\cR\backslash\cs^*|-2) |\cR\backslash\cs^*|, & \text{ $|\cs^*|$ even}\\
(\delta_{\cs^*}-t)  (|\cs^*|-1)^2 + (\delta_{\cR\backslash\cs^*}+t) (|\cR\backslash\cs^*|-1)^2, & \text{ $|\cs^*|$ odd}
\end{cases}\\
=&\sum_{\substack{|\cs| \textrm{ even}\\ \cs \neq \cR}} \delta_{\cs} (|\cs|-2)|\cs| + \sum_{\substack{|\cs| \textrm { odd}\\ \cs \neq \cR}} \delta_{\cs}(|\cs|-1)^2\\ &~~~~~~ +
\begin{cases} -t  (|\cs^*|-2)|\cs^*|+ {t}(2g- |\cs^*|)(2g+2-|\cs^*|), & \text{$|\cs^*|$ even}\\
-t(|\cs^*|-1)^2 + {t}(2g + 1 - |\cs^*|)^2, & \text{$|\cs^*|$ odd}
\end{cases}\\
 =& \sum_{\substack{|\cs| \textrm{ even}\\ \cs \neq \cR}} \delta_{\cs} (|\cs|-2)|\cs| + \sum_{\substack{|\cs| \textrm { odd}\\ \cs \neq \cR}} \delta_{\cs}(|\cs|-1)^2 + t (|\cR|-2)(|\cR|-2|\cs|)
\end{align*} 

Now the only condition missing in order to apply Lemma \ref{lem:kausz} is $\cR \subset R$. We already have $\cR \subset K$. Since $d_{\cR} =0$ and $v(c_f) = 0$, it follows from \cite[Theorem 13.3]{dokchitser2018arithmetic} that there exists $z \in K$ such that $f(x-z) \in R[x]$. Clearly, such a shift changes neither the cluster picture nor the valuation of $\lambda$. Hence we may assume $\cR \subset R$ without adding a further correction term.
\end{proof}

\section{A Basis for $H^0(\regX,\omega_{\regX/R})$}
\label{sec:base}

Let $C/K$ be a semistable hyperelliptic curve defined by a Weierstraß equation $C: y^2=f(x)$ with $f(x) = c_f\prod_{r\in \cR}(x-r)$. To this equation we associate the cluster picture $\Sigma$. Let $\regX \rightarrow \Spec R$ be the minimal regular model of $C$.
In this section we show how to read off a basis for the global sections of the canonical sheaf $\omega_{\regX/R}$ from the cluster picture $\Sigma$.
Our result is valid in a more general setting than \cite[Proposition 5.5.1.]{kausz1999discriminant} and at the same time simplifies the construction of the basis described therein.

\begin{thm}
	\label{thm:mainbasis}
Let $C/K$ be a semistable hyperelliptic curve defined by an integral Weierstraß equation $C:y^2=f(x)$ with cluster picture $\Sigma$.
Let $\regX/ R$ be the minimal regular model. Assume that the residue field $k$ is algebraically closed. 

Choose clusters $\cs_0, \dots, \cs_{g-1}$  inductively such that
\[e_i = \frac{\nu_{\cs_i}}{2} - \sum_{j=0}^{i}d_{\cs_j \land \cs_i} = \max \left( \frac{\nu_{\cs}}{2} - \sum_{j=0}^{i-1}d_{\cs_j \land \cs} -d_{\cs}\right),\]
where the maximum is taken over all proper clusters in $\Sigma$.
If the maximal value is obtained by two different clusters $\cs$ and $\cs'$ with $\cs' \subset \cs$, choose  $\cs_i = \cs$. 

Then an $R$-basis for the global sections of the relative dualising sheaf $\omega_{\regX/R}$ is given by $(\mu_0, \dots \mu_{g-1})$, where
\[\mu_i = \pi^{e_i} \prod_{j=0}^{i-1}(x-z_{\cs_{j}}) \frac{dx}{2y}.\]
\end{thm}

Note that the same cluster can appear multiple times in the sequence $\cs_0, \dots, \cs_{g-1}$. Moreover the sequence is not canonical, since it can happen that the maximum is obtained by two incomparable clusters, in which case any of the two clusters may be chosen. 
Since $k$ is algebraically closed, one can find a centre $z_{\cs} \in K$ for every proper cluster $\cs$. This follows from \cite[Lemma 4.2.]{dokchitser2018arithmetic}. For an illustration of the theorem, we refer to Example \ref{ex:basis} in the introduction. 

The strategy for the proof is a follows. First, we show that the sum over the $e_i$ is exactly the valuation of $\lambda_C$ as defined in the previous section. For this we need Lemma \ref{lem:gamma}. Then, we show that the differential forms are indeed global sections by computing their order of vanishing along different components of the special fibre of the minimal regular models, see Lemma \ref{lem:e_i}. The theorem then follows as a corollary of Theorem \ref{thm:main_lambda}.

\begin{lem}
	\label{lem:gamma}
	Let $C/K$ be a hyperelliptic curve defined by an integral Weierstraß equation $C:y^2=f(x)$ and $\Sigma$ the associated cluster picture. Let the clusters $\cs_0, \dots, \cs_{g-1} \in \Sigma$ be chosen according to Theorem \ref{thm:mainbasis}. 
	Then for every cluster $\cs \in \Sigma$:
	\[\gamma(\cs) := \# \{\cs_i \,|\, \cs_i \subset \cs\} = \left \lfloor{\frac{|\cs|-1}{2}}\right \rfloor\]
\end{lem}

\begin{proof}
 Since $\gamma(\cR) = g$, the statement is true for $\cs = \cR$. Let $\cs \neq \cR$ and let $\cs'$ be the parent of $\cs$, that is $\cs < \cs'$. Then 
	\begin{align*}
		d_{\cs} =& d_{\cs'} + \delta_{\cs},\\
		\nu_{\cs} =& \nu_{\cs'}+\delta_{\cs}|\cs|,\\ 
		d_{\cs_j\land\cs} =& \begin{cases}
			d_{\cs_j \land \cs'} + \delta_{\cs} &\textrm{ if } \cs_j \subset \cs\\
			d_{\cs_j \land \cs'}  &\textrm{otherwise}
		\end{cases} ~~\textrm{ for }j \in \{0, \dots, i-1\}.
	\end{align*}
	So for any $i \in \{0, \dots ,g-1\}$
	\[\frac{\nu_{\cs}}{2} - \sum_{j=0}^{i-1}d_{\cs_j \land \cs} -d_{\cs} =  \frac{\nu_{\cs'}}{2} - \sum_{j=0}^{i-1}d_{\cs_j \land \cs'} -d_{\cs'} + \delta_{\cs} \left(\frac{|\cs|}{2} - 1 -\# \{\cs_j \,|\, \cs_j \subset \cs, \;0 \leq j \leq i-1\}\right).\]
	
	Assume that $\cs_i = \cs$, then by construction of the sequence $(\cs_j)$, it must hold that $\frac{\nu_{\cs}}{2} - \sum_{j=0}^{i-1}d_{\cs_j \land \cs} -d_{\cs} >  \frac{\nu_{\cs'}}{2} - \sum_{j=0}^{i-1}d_{\cs_j \land \cs'} -d_{\cs'}$. The equation above then implies $ \# \{\cs_j \,|\, \cs_j \subset \cs, \;0 \leq j \leq i-1\} < \frac{|\cs|}{2} - 1$. Hence $\{\cs_j \,|\, \cs_j \subset \cs, \;0 \leq j \leq i\} \leq \frac{|\cs| -1}{2}$. Since this is true for every $i$, we may conclude $\gamma(\cs) \leq \lfloor \frac{|\cs|-1}{2} \rfloor$.
	
	To show equality, we proceed by induction. Again let $\cs \neq \cR$ and assume that the statement holds for every cluster strictly containing $\cs$. If $\gamma(\cs) < \lfloor \frac{|\cs|-1}{2} \rfloor$, then at any step $i$, we have $\frac{\nu_{\cs}}{2} - \sum_{j=0}^{i-1}d_{\cs_j \land \cs} -d_{\cs} >  \frac{\nu_{\cs'}}{2} - \sum_{j=0}^{i-1}d_{\cs_j \land \cs'} -d_{\cs'}$, hence $\cs_i \neq \cs'$. This implies $\gamma(\cs') = \sum_{\ct < \cs'}\gamma(\ct)$. Using that $|\cs'| = \sum_{\ct < \cs'}|\ct|$, $\gamma(\ct)\leq  \lfloor \frac{|\ct|-1}{2} \rfloor$ for every $\ct < \cs'$ and $\gamma(\cs) <\lfloor \frac{|\cs|-1}{2} \rfloor$, we get
	\[
	\gamma(\cs') = \sum_{\ct < \cs'} \gamma(\ct) 
	< \left\lfloor \frac{|\cs|-1}{2} \right\rfloor + \sum_{\cs \neq \ct < \cs'} \left\lfloor \frac{|\ct|-1}{2} \right\rfloor 
	\leq  \left\lfloor \frac{\sum_{\ct < \cs'} |\ct| - 1}{2} \right\rfloor 
	= \left\lfloor \frac{|\cs'|-1}{2} \right\rfloor.
	\]
	 This yields a contradiction. So $\gamma(\cs) = \lfloor \frac{|\cs|-1}{2} \rfloor$.  
	
\end{proof}

In order to prove Theorem \ref{thm:mainbasis}, we have to show that the differential forms defined in the theorem are global sections of $\omega_{\regX/R}$. In particular, we need to be able to compute the order of vanishing along components of the special fibre of the minimal regular model $\regX$. For that purpose we use the description of $\regX$ given in [Theorem 8.5]\cite{dokchitser2018arithmetic}. Broadly speaking, each principal cluster gives rise to one or possibly two components in the special fibre. These components are connected by chains of projective lines. In the following, we will write $\Gamma_{\cs} \subset \regX_s$ for the component(s) corresponding to the principal cluster $\cs$. 
On the other hand, a disc $D(\cs) \subset \bar{K}$ is associated to a principal cluster $\cs$. Equations for $\Gamma_{\cs}$ are now given in function of the disc $D(\cs)$ in \cite[Definition 5.4]{dokchitser2018arithmetic}. From this, we can conclude that there is an open neighbourhood $U_{\cs}$ of the generic point of $\Gamma_{\cs}$ which is contained in 
\[ 
\Spec\left(R[x_{\cs},y_{\cs}]/ (y_{\cs}^2 - f_{\cs}(x_{\cs})) \right) 
\] 
with local coordinates
\[
x_{\cs} = \frac{x-z_{\cs}}{\pi^{d_\cs}}, \;
y_{\cs} = \frac{y}{\pi^{\nu_{\cs}/2}}.
\]

\begin{lem}
	\label{lem:e_i}
	Let $C/K$ be a hyperelliptic curve defined by an integral Weierstraß equation $C:y^2=f(x)$ and $\Sigma$ the associated cluster picture. Let the clusters $\cs_0, \dots, \cs_{g-1} \in \Sigma$ be chosen according to Theorem \ref{thm:mainbasis}. Let $\regX/R$ be the minimal regular model.

	Then for any principal cluster $\cs \in \Sigma$, we have that  $-\left(\frac{\nu_{\cs}}{2} - \sum_{j=0}^{i-1}d_{\cs_j \land \cs} -d_{\cs}\right)$ is a lower bound for the order of vanishing of the element $\prod_{j=0}^{i-1}(x-z_{\cs_{j}}) \frac{dx}{2y}$ along the component of the special fibre of $\regX$ that corresponds to $\cs$.
\end{lem}

\begin{proof}
	Let $\cs$ be a principal cluster in $\Sigma$ and denote by $\Gamma_{\cs}$ the component (or possibly the two components) of the special fibre of $\regX/R$ corresponding to this cluster.
	
	It suffices to compute the order of vanishing of $\prod_{j=0}^{i-1}(x-z_{\cs_{j}}) \frac{dx}{2y}$ in an open neighbourhood $U$ of the generic point of the component $\Gamma_{\cs}$. We choose $U = U_{\cs}$ as described in the paragraph preceding this lemma. 
	
	We can write
	\begin{align*}
	\prod_{j=0}^{i-1}(x-z_{\cs_{j}}) \frac{dx}{2y} &= \prod_{j=0}^{i-1}(\pi^{d_{\cs}} x_{\cs}+z_{\cs}-z_{\cs_{j}}) \cdot \frac{d(\pi^{d_{\cs}}x_{\cs} +z_{\cs})}{2\pi^{\nu_{\cs}/2}y_{\cs}}\\
	&=\prod_{j=0}^{i-1}\pi^{d_{\cs \land \cs_j}} \cdot (\pi^{d_{\cs}-d_{\cs \land \cs_j}} x_{\cs} - \frac{z_{\cs}-z_{\cs_{j}}}{\pi^{d_{\cs \land \cs_j}}}) \cdot \frac{\pi^{d_{\cs}}}{\pi^{\nu_{\cs}/2}} \frac{dx_{\cs}}{2y_{\cs}}\\
	&=\pi^ {\sum_{j=0}^{i-1}d_{\cs \land \cs_j} + d_{\cs} - \frac{\nu_{\cs}}{2}} \cdot \prod_{j=0}^{i-1} (\pi^{d_{\cs}-d_{\cs \land \cs_j}} x_{\cs} - \frac{z_{\cs}-z_{\cs_{j}}}{\pi^{d_{\cs \land \cs_j}}}) \cdot \frac{dx_{\cs}}{2y_{\cs}}.
	\end{align*}
	Since  $d_{\cs} \geq d_{\cs \land \cs_j}$ and $\frac{z_{\cs}-z_{\cs_{j}}}{\pi^{d_{\cs \land \cs_j}}} \in R$ for all clusters $\cs, \cs_j \in \Sigma$, the element $\prod_{j=0}^{i-1} (\pi^{d_{\cs}-d_{\cs \land \cs_j}} x_{\cs} - \frac{z_{\cs}-z_{\cs_{j}}}{\pi^{d_{\cs \land \cs_j}}}) \frac{dx_{\cs}}{2y_{\cs}}$ is an integral section on $\omega_{\regX/R|_{U}}$. The statement of the lemma follows.

\end{proof}

\begin{proof}[Proof of Theorem \ref{thm:mainbasis}]
	Let $\cs_0, \dots, \cs_{g-1}$ be a sequence constructed as described in the theorem and denote by $\mu_0, \dots, \mu_{g-1}$ the differential forms associated to this sequence.
	
	\textbf{Claim 1:} The differentials $\mu_0, \dots \mu_{g-1}$ are  global sections of $\omega_{\regX/R}$.\\	
	Here, we are going to make use of the fact that inside $H^0(C, \Omega_{C/K})$ the global sections of the minimal regular model are equal to those of the stable model (cf. the proof of Proposition \ref{prop:H0basechange}). So by \cite[Theorem 5.24.]{dokchitser2018arithmetic}, it suffices to check that the differentials $\mu_0, \dots ,\mu_{g-1}$ are regular on the components of the special fibre corresponding to principal clusters.
	Let $\Gamma_{\cs}$ be a component of the special fibre corresponding to a principal cluster $\cs \in \Sigma$. From Lemma \ref{lem:e_i}, we know that  $-\left(\frac{\nu_{\cs}}{2} - \sum_{j=0}^{i-1}d_{\cs_j \land \cs} -d_{\cs}\right)$ is a lower bound for the order of the element $\prod_{j=0}^{i-1}(x-z_{\cs_{j}}) \frac{dx}{2y}$ on $\Gamma_{\cs}$. Since $e_i = \max_{\cs \in \Sigma} \left( \frac{\nu_{\cs}}{2} - \sum_{j=0}^{i-1}d_{\cs_j \land \cs} -d_{\cs}\right)$, the order of $\mu_i = \pi^{e_i} \cdot \prod_{j=0}^{i-1}(x-z_{\cs_{j}}) \frac{dx}{2y}$ is non-negative on every such component.
	
	For the horizontal part, we have to consider the restriction of $\mu_i$ to the generic fibre. Clearly $\mu_i \otimes 1 \in H^0(\regX,\omega_{\regX/R}) \otimes_R K = H^0(C,\Omega_{C/K})$. This proves the first claim. 
	
	\textbf{Claim 2:} Let $\lambda_C$ be the quantity defined in the previous section, then
	$\sum_{i=0}^{g-1} e_i = v(\lambda_C)$.
	By definition we have 
	\[2 \sum_{i=0}^{g-1} e_i~ 
	= \sum_{i=0}^{g-1}\left(\nu_{\cs_{i}}- 2 \sum_{j=0}^{i} d_{\cs_j \land \cs_i}\right).
	\]
	We divide this expression into two parts. For the first part, we use the formula from Remark \ref{rem:compute_vs} and get
	\begin{align*}
	 \sum_{i=0}^{g-1}\nu_{\cs_{i}}
	=&\; \sum_{i=0}^{g-1} \left(v(c_f) +
	d_{\cR} |\cR| +  \sum_{\cs \supset \cs_i} \delta_{\cs} |\cs|\right) 
	=g  (v(c_f) + d_{\cR} |\cR|) + \sum_{\cs \neq \cR} \delta_{\cs} |\cs| \gamma(\cs).
	\end{align*}
	The second part gives 
	\begin{align*}
	2 \sum_{i=0}^{g-1}\sum_{j=0}^{i} d_{\cs_j \land \cs_i} 
	= 2\sum_{i\geq j} \left(\sum_{\cs \supset \cs_j\land \cs_i} \delta_{\cs} +d_{\cR} \right)
	=  \sum_{\cs \neq \cR} \delta_{\cs} \gamma(\cs)(\gamma(\cs)+1) + d_{\cR} \gamma(\cR)(\gamma(\cR) + 1).
	\end{align*}
	For the first equality, we used the fact that the depth of a proper cluster $\ct$ may be written as $d_{\ct} = \sum_{\cs \supset \ct} \delta_{\cs} + d_{\cR}$.  For the second equality, we used that  $ \#\{\cs_j\land \cs_i \subseteq \cs~|~ j\leq i\} = \binom{\gamma(\cs)+1}{2}$ for every proper cluster $\cs$. By Lemma \ref{lem:gamma}, we have $\gamma(\cs) = \lfloor \frac{|\cs|-1}{2}\rfloor$.
	Combining the two parts, we obtain 
	\begin{align*}
	8  \sum_{i=0}^{g-1} e_i~ 
	=&~~ 4g \cdot v(c_f) + \sum_{\substack{|\cs| \textrm{ even}\\ \cs \neq \cR}} \delta_{\cs} \cdot (|\cs|-2)|\cs| + \sum_{\substack{|\cs| \textrm { odd},\\ \cs \neq \cR}} \delta_{\cs}(|\cs|-1)^2 \\ 
	&~~~~~~ + d_{\cR} \begin{cases} (|\cR|-2)|\cR|, & \text{if $|\cR| = 2g+2$},\\
	(|\cR|-1)^2, & \text{if $|\cR| = 2g+1$}.
	\end{cases}	
	\end{align*}
	The claim now follows from Theorem \ref{thm:main_lambda}.
	
	\textbf{Claim 3:}
	$\mu: = \mu_0 \land \dots \land \mu_{g-1}$ is a basis for $\det H^0(\regX,\omega_{\regX/R})$.\\	
	Let $\omega_0, \dots ,\omega_{g-1}$ denote the differentials associated to the Weierstraß equation $C: y^2 = f(x)$. By construction, we have that 
	\[(\mu_0,\dots,\mu_{g-1}) = (\omega_0, \dots ,\omega_{g-1}) \cdot A, \]
	where $A$ is a matrix of the form  
	\(  A = \begin{pmatrix}
	\pi^{e_0}	&  & \ast    \\
	& \ddots  & \\
	0 	& 	 & \pi^{e_{g-1}}
	\end{pmatrix}.\)\\
	Therefore \[\mu = \det(A) \cdot  \omega_0 \land \dots \land \omega_{g-1} =  \pi^{v(\lambda_C)} \cdot \omega_0 \land \dots \land \omega_{g-1}.\] So by Theorem \ref{thm:main_lambda}, $\mu$ is a basis for $\det H^0(\regX,\omega_{\regX/R})$.\\
	
	We have seen that $\mu_0,\dots,\mu_{g-1} \in H^0(\regX,\omega_{\regX/R})$ and that $\mu: = \mu_0 \land \dots \land \mu_{g-1}$ is a basis for $\det H^0(\regX,\omega_{\regX/R})$. Therefore $(\mu_0,\dots,\mu_{g-1})$ is a basis for $H^0(\regX,\omega_{\regX/R})$. 
	
\end{proof}


\begin{thebibliography}{1}
	
	\bibitem{bouw2017computing}
	Irene~I. Bouw and Stefan Wewers.
	\newblock Computing {L}-functions and semistable reduction of superelliptic
	curves.
	\newblock {\em Glasgow Mathematical Journal}, 59(1):77--108, 2017.
	
	\bibitem{dokchitser2018arithmetic}
	Tim Dokchitser, Vladimir Dokchitser, C{\'e}line Maistret, and Adam Morgan.
	\newblock Arithmetic of hyperelliptic curves over local fields.
	\newblock {\em arXiv preprint arXiv:1808.02936}, 2018.
	
	\bibitem{flynn2001empirical}
	Victor Flynn, Franck Lepr{\'e}vost, Edward Schaefer, William Stein, Michael Stoll,
	and Joseph Wetherell.
	\newblock Empirical evidence for the {B}irch and {S}winnerton-{D}yer
	conjectures for modular jacobians of genus 2 curves.
	\newblock {\em Mathematics of Computation}, 70(236):1675--1697, 2001.
	
	\bibitem{kausz1999discriminant}
	Ivan Kausz.
	\newblock A discriminant and an upper bound for $\omega^2$ for hyperelliptic
	arithmetic surfaces.
	\newblock {\em Compositio Mathematica}, 115(1):37--69, 1999.
	
	\bibitem{liu2002alggeo}
	Qing Liu.
	\newblock {\em Algebraic Geometry and Arithmetic of curves}, volume~6.
	\newblock Oxford University Press, 2002.
	
	\bibitem{tate1966bsd}
	John Tate.
	\newblock On the conjectures of {B}irch and {S}winnerton-{D}yer and a geometric
	analog.
	\newblock In {\em S\'eminaire Bourbaki : ann\'ees 1964/65 1965/66, expos\'es
		277-312}, number~9 in S\'eminaire Bourbaki, pages 415--440. Soci\'et\'e
	math\'ematique de France, 1964-1966.
	\newblock talk:306.
	\newblock URL: \url{http://www.numdam.org/item/SB_1964-1966__9__415_0}.
	
	\bibitem{van2017numerical}
	Raymond van Bommel.
	\newblock Numerical verification of the {B}irch and {S}winnerton-{D}yer
	conjecture for hyperelliptic curves of higher genus over $\mathbb{Q} $ up to
	squares.
	\newblock {\em Experimental Mathematics}, 1--8, 2019.
	
\end{thebibliography}
\end{document}